\documentclass{article}
\usepackage{amssymb}
\usepackage{latexsym}
\usepackage{amsmath}
\usepackage{amsthm}
\topskip  \theoremstyle{remark}
\theoremstyle{plain}
\newtheorem{theorem}{Theorem}[section]

\newtheorem{definition}[theorem]{Definition}

\newtheorem{corollary}[theorem]{Corollary}

\newcommand{\stirling}[2]{\genfrac{[}{]}{0pt}{}{#1}{#2}}
\newcommand{\Stirling}[2]{\genfrac{\{}{\}}{0pt}{}{#1}{#2}}

\newcommand{\Lah}[2]{\genfrac{\lfloor}{\rfloor}{0pt}{}{#1}{#2}}
\begin{document}
\title{Generalized $r$-Lah numbers}
\date{\small }
\maketitle

\begin{center}Mark Shattuck\\
Department of Mathematics, University of Tennessee, 37996 Knoxville, TN, USA

{\tt shattuck@math.utk.edu}
\end{center}


\begin{abstract}
In this paper, we consider a two-parameter polynomial generalization, denoted by  $\mathcal{G}_{a,b}(n,k;r)$, of the
$r$-Lah numbers which reduces to these recently introduced numbers when $a=b=1$.  We present several identities for $\mathcal{G}_{a,b}(n,k;r)$ that generalize earlier identities given for the $r$-Lah and $r$-Stirling numbers.  We also provide combinatorial proofs of some identities involving the $r$-Lah numbers which were established previously using algebraic methods.  Generalizing these arguments yields orthogonality-type relations that are satisfied by $\mathcal{G}_{a,b}(n,k;r)$.

\end{abstract}



\section{Introduction}

Let $[m]=\{1,2,\ldots,m\}$ if $m \geq 1$, with $[0]=\varnothing$.  By a \emph{partition} of $[m]$, we will mean a collection of non-empty, pairwise disjoint subsets of $[m]$, called \emph{blocks}, whose union is $[m]$.  A \emph{Lah distribution} will refer to a partition of $[m]$ in which elements within each block are ordered (though there is no inherent ordering for the blocks themselves).  Following the notation from \cite{NR}, let $\Lah{n}{k}_r$ denote the number of Lah distributions of the elements of $[n+r]$ having $k+r$ blocks such that the elements of $[r]$ belong to distinct (ordered) blocks.  The $\Lah{n}{k}_r$ are called $r$-\emph{Lah numbers} and have only recently been studied.

The numbers $\Lah{n}{k}_r$ were once mentioned in \cite{Ch} (together with $r$-Whitney-Lah numbers) and appear in \cite{Sl} under the name of restricted Lah numbers.  A few properties of the $\Lah{n}{k}_r$ were established in \cite{Be}, and a systematic study of these numbers was undertaken in \cite{NR}.  When $r=0$, the $r$-Lah number reduces to the Lah number $\Lah{n}{k}$ (named for the mathematician Ivo Lah \cite{La1} and often denoted by $L(n,k)$), which counts the number of partitions of $[n]$ into $k$ ordered blocks (see, e.g., \cite{BeB,Wa}).

Earlier, analogous $r$-versions of the Stirling numbers of the first and second kind were introduced by Broder \cite{Br}, and later rediscovered by Merris \cite{Mer}, where $r$ distinguished elements have to be in distinct cycles or blocks.  Following the parametrization and notation used in \cite{NR}, let $\stirling{n}{k}_r$ be the number of permutations of $[n+r]$ into $k+r$ cycles in which members of $[r]$ belong to distinct cycles and let $\Stirling{n}{k}_r$ be the number of partitions of $[n+r]$ into $k+r$ blocks in which members of $[r]$ belong to distinct blocks.  There are several algebraic properties for which $\stirling{n}{k}_r$ and $\Stirling{n}{k}_r$ satisfy analogous identities, among them various recurrences and connection constant relations (see \cite[Section 2]{NR} for a comparative study).  Analogues of some of these properties involving $r$-Lah numbers were established in \cite{NR}.

In this paper, we consider a two-parameter polynomial generalization of the number $\Lah{n}{k}_r$ which reduces to it when both parameters are unity.  Denoted by $\mathcal{G}_{a,b}(n,k;r)$, it will also be seen to specialize to $\stirling{n}{k}_r$ when $a=1,b=0$ and to $\Stirling{n}{k}_r$ when $a=0,b=1$.  Since the $\mathcal{G}_{a,b}(n,k;r)$ reduce to the $r$-Lah numbers when $a=b=1$, we will refer to them as \emph{generalized} $r$-\emph{Lah numbers}.  We note that the special case $r=0$ is equivalent to a re-parametrized version of the numbers $\mathfrak{S}_{s;h}(n,k)$ (see \cite{MSS,MSS2}) which arise in conjunction with the normal ordering problem from mathematical physics.  Furthermore, it is seen that $\mathcal{G}_{a,b}(n,k;r)$ is a variant of the generalized Stirling polynomial introduced by Hsu and Shiue in \cite{HS} and studied from an algebraic standpoint. Finally, we remark that $\mathcal{G}_{a,b}(n,k;r)$ can be reached by a special substitution into the partial $r$-Bell polynomials introduced in \cite{MiR}, but here we look at some specific properties of $\mathcal{G}_{a,b}(n,k;r)$ that were not considered more generally in \cite{MiR} (several of which do not seem to hold in the more general setting).

The paper is organized as follows.  We define the polynomial $\mathcal{G}_{a,b}(n,k;r)$ in terms of two statistics on Lah distributions and derive several of its properties, mostly by combinatorial arguments.  This furnishes a common generalization for several of the earlier identities proven for $r$-Stirling and $r$-Lah numbers (see \cite{Br,Me,NR}).  Some recurrences are also given for the row sum $\sum_{k=0}^n \mathcal{G}_{a,b}(n,k;r)$, among them a generalization of Spivey's Bell number formula \cite{Sp1}.

In the third section, we provide combinatorial proofs of four identities involving $\Lah{n}{k}_r$, which were shown in \cite{NR} by algebraic arguments.  To do so, in three cases, we define appropriate sign-changing involutions on certain ordered pairs of combinatorial configurations, while in the other case, a suitable bijection is defined between the ordered pairs and a subset of the Lah distributions.  Modifying the proofs will give orthogonality-type relations satisfied by the generalized $r$-Lah numbers.

We will make use of the following notation and conventions.  Empty sums will take the value zero, and empty products the value one.  The binomial coefficient $\binom{n}{k}$ is defined as $\frac{n!}{k!(n-k)!}$ if $0 \leq k \leq n$, and will be taken to be zero otherwise.  If $m$ and $n$ are positive integers, then $[m,n]=\{m,m+1,\ldots,n\}$ if $m \leq n$, with $[m,n]=\varnothing$ if $m>n$.  Finally, if $n$ is an integer, then let  $n^{\overline{m}}=\prod_{i=0}^{m-1} (n+i)$ and $n^{\underline{m}}=\prod_{i=0}^{m-1} (n-i)$ if $m \geq 1$, with $n^{\overline{0}}=n^{\underline{0}}=1$ for all $n$.

\section{Generalized $r$-Lah relations}

Given $0 \leq k \leq n$ and $r \geq 0$, let $\mathcal{L}_r(n,k)$ denote the set of Lah distributions enumerated by $\Lah{n}{k}_r$, i.e., partitions of $[n+r]$ into $k+r$ ordered blocks in which the elements of $[r]$ belong to distinct blocks.  We will say that the elements of $[r]$ within a member of $\mathcal{L}_r(n,k)$ are \emph{distinguished} and apply this term also to the blocks in which they belong.  We will sometimes refer to the members of $\mathcal{L}_r(n,k)$ as $r$-\emph{Lah distributions}.

Note that when $r=0$ or $r=1$, there is no restriction introduced by distinguished elements so that $\Lah{n}{k}_0=\Lah{n}{k}$ and $\Lah{n}{k}_1=\Lah{n+1}{k+1}$.  Accordingly, when $r=0$, we will often omit the subscript and let $\mathcal{L}(n,k)$ denote the set of all Lah distributions of size $n$ having $k$ blocks.  Note that $\mathcal{L}_r(n,k)$ is a proper subset of $\mathcal{L}(n+r,k+r)$ when $r \geq 2$ and $n>k$.

We consider a generalization of the numbers $\Lah{n}{k}_r$ obtained by introducing a pair of statistics on $\mathcal{L}_r(n,k)$ as follows.  If $\lambda \in \mathcal{L}(n,k)$ and $i \in [n]$, then we will say that $i$ is a \emph{record low} of $\lambda$ if there are no elements $j<i$ to the left of $i$ within its block in $\lambda$.  For example, if $n=9$, $k=3$ and $\lambda=\{1,5,3\},\{8,4,7,2,9\},\{6\} \in \mathcal{L}(9,3)$, then the element 1 is a record low in the first block, 8, 4 and 2 are record lows in the second, and 6 is a record low in the third block for a total of five record lows altogether.  Note that the first element within a block as well as the smallest are always record lows.

We now recall the following statistic from \cite{MSS}.

\begin{definition}\label{d1}
Given $\lambda \in \mathcal{L}(n,k)$, let $rec^*(\lambda)$ denote the total number of record lows of $\lambda$ which are not themselves the smallest element of a block.  Let $nrec(\lambda)$ denote the number of elements of $[n]$ which are not record lows of $\lambda$.
\end{definition}

To illustrate, if $\lambda$ is as above, then $rec^*(\lambda)=2$ (for the 8 and 4) and $nrec(\lambda)=4$ (for 5, 3, 7 and 9).  We now consider the restriction of the $rec^*$ and $nrec$ statistics to $\mathcal{L}_r(n,k)$ and define the distribution polynomial $\mathcal{G}_{a,b}(n,k;r)$ by
$$\mathcal{G}_{a,b}(n,k;r)=\sum_{\lambda \in \mathcal{L}_r(n,k)}a^{nrec(\lambda)}b^{rec^*(\lambda)},$$
where $a$ and $b$ are indeterminates.

Note that $\mathcal{G}_{a,b}(n,k;r)$ reduces to $\Lah{n}{k}_r$ when $a=b=1$, by definition.  Furthermore, it is seen that $\mathcal{G}_{a,b}(n,k;r)$ reduces to $\stirling{n}{k}_r$ when $a=1,b=0$ and to $\Stirling{n}{k}_r$ when $a=0,b=1$.  Note that in the former case, the first element must be the smallest within each block in order for $\lambda \in \mathcal{L}_r(n,k)$ to have a non-zero contribution towards $\mathcal{G}_{a,b}(n,k;r)$, while in the latter case, the elements must be arranged in decreasing order within each block of $\lambda$.

Given $\lambda \in \mathcal{L}(n,k)$, let $w(\lambda)=a^{nrec(\lambda)}b^{rec^*(\lambda)}$ denote the \emph{weight of} $\lambda$, and by the weight of a subset of $\mathcal{L}(n,k)$, we will mean the sum of the weights of all the members contained therein.  Note that $\mathcal{G}_{a,b}(n,k;r)$ can only assume non-zero values when $0 \leq k \leq n$ and $r \geq 0$.   We now write a recurrence for $\mathcal{G}_{a,b}(n+1,k;r)$ where $1 \leq k \leq n+1$.  First note that the total weight of all members of $\mathcal{L}_r(n+1,k)$ in which the element $n+r+1$ belongs to its own block is $\mathcal{G}_{a,b}(n,k-1;r)$ since $n+r+1$ in this case contributes to neither the $nrec$ nor $rec^*$ values.  The weight of all members of $\mathcal{L}_r(n+1,k)$ in which $n+r+1$ starts a block containing at least one member of $[n+r]$ is $b(k+r)G_{a,b}(n,k;r)$ since $rec^*$ is increased by one by the addition of $n+r+1$.  Finally, if $n+r+1$ directly follows some member of $[n+r]$ within a block, then $nrec$ is increased by one, which implies a contribution of $a(n+r)G_{a,b}(n,k;r)$ in this case.  Combining the three previous cases gives the recurrence
\begin{equation}\label{rec1}
\mathcal{G}_{a,b}(n+1,k;r)=\mathcal{G}_{a,b}(n,k-1;r)+(an+bk+(a+b)r)\mathcal{G}_{a,b}(n,k;r), \qquad 1 \leq k \leq n+1,
\end{equation}
with boundary values $\mathcal{G}_{a,b}(0,k;r)=\delta_{k,0}$ and $\mathcal{G}_{a,b}(n,0;r)=\prod_{i=0}^{n-1}(a(i+r)+br)$.

\textbf{Remark:} By \eqref{rec1}, one sees that the $\mathcal{G}_{a,b}(n,k;r)$ occur as a special case of the solution to a general bivariate recurrence in \cite{BSV1,BSV2}, which was approached algebraically (wherein general formulas for the relevant exponential generating functions were found).

The $a=b=1$ case of the following result occurs as \cite[Theorem 3.2]{NR}.

\begin{theorem}\label{gt1}
If $n\geq 0$, then
\begin{equation}\label{gt1e1}
\prod_{i=0}^{n-1}(x+(a+b)r+ai)=\sum_{k=0}^n \mathcal{G}_{a,b}(n,k;r)\prod_{i=0}^{k-1}(x-bi).
\end{equation}
\end{theorem}
\begin{proof}
Proceed by induction on $n$, the $n=0$ case clear.  If $n \geq 0$, then
\begin{align*}
\prod_{i=0}^n(x+(a+b)r+ai)&=(x+(a+b)r+an)\prod_{i=0}^{n-1}(x+(a+b)r+ai)\\
&=(x+(a+b)r+an)\sum_{k=0}^n \mathcal{G}_{a,b}(n,k;r)\prod_{i=0}^{k-1}(x-bi)\\
&=\sum_{k=0}^n \mathcal{G}_{a,b}(n,k;r)\left[\prod_{i=0}^k(x-bi)+(an+bk+(a+b)r)\prod_{i=0}^{k-1}(x-bi)\right]\\
&=(an+(a+b)r)\mathcal{G}_{a,b}(n,0;r)+\sum_{k=1}^{n+1}\mathcal{G}_{a,b}(n,k-1;r)\prod_{i=0}^{k-1}(x-bi)\\
\end{align*}
\begin{align*}
&\quad\quad\quad+\sum_{k=1}^{n+1}(an+bk+(a+b)r)\mathcal{G}_{a,b}(n,k;r)\prod_{i=0}^{k-1}(x-bi)\\
&\quad\quad=\sum_{k=0}^{n+1} G_{a,b}(n+1,k;r)\prod_{i=0}^{k-1}(x-bi),
\end{align*}
by \eqref{rec1}, which completes the induction.
\end{proof}

The $a=b=1$ case of \eqref{gt2e1} below occurs as \cite[Theorem 3.3]{NR}.

\begin{theorem}\label{gt2}
We have
\begin{equation}\label{gt2e1}
\mathcal{G}_{a,b}(n,k;r)=\sum_{i=k}^n\mathcal{G}_{a,b}(i-1,k-1;r)\prod_{j=i}^{n-1}(aj+bk+(a+b)r), \qquad 1 \leq k \leq n,
\end{equation}
and
\begin{equation}\label{gt2e2}
\mathcal{G}_{a,b}(n,k;r)=\sum_{i=0}^{k}(a(n+r-i-1)+b(k+r-i))\mathcal{G}_{a,b}(n-i-1,k-i;r), \qquad 0 \leq k <n.
\end{equation}
\end{theorem}
\begin{proof}
To show \eqref{gt2e1}, we may assume that the blocks within an $r$-Lah distribution are arranged from left to right in ascending order according to the size of the smallest element.  Then the right-hand side of \eqref{gt2e1} gives the total weight of all members of $\mathcal{L}_r(n,k)$ by considering the smallest element, $i+r$, belonging to the right-most block where $k \leq i \leq n$.  Note that there are $\mathcal{G}_{a,b}(i-1,k-1;r)$ possibilities concerning placement of the members of $[i+r-1]$ and $\prod_{j=i}^{n-1}(aj+bk+(a+b)r)$ ways in which to arrange the members of $[i+r+1,n+r]$.  Summing over all possible $i$ gives \eqref{gt2e1}.

To show \eqref{gt2e2}, consider the largest element, $n+r-i$, not going by itself in a block where $0 \leq i \leq k$ (note $k<n$ implies the existence of such an element).  Observe that then the elements of $[n+r-i-1]$ comprise a member of $\mathcal{L}_r(n-i-1,k-i)$ and that there are $a(n+r-i-1)+b(k+r-i)$ possibilities concerning placement of $n+r-i$.  Finally, the members of $[n+r-i+1,n+r]$ must all belong to singleton blocks and hence contribute to neither the $nrec$ nor the $rec^*$ values.
\end{proof}

Extending the proofs of Theorems 3.4 and 3.6 in \cite{NR} yields the following identities.

\begin{theorem}\label{gt3}
If $0 \leq k \leq n$, then
\begin{equation}\label{gt3e1}
\mathcal{G}_{a,b}(n,k;r+s)=\sum_{i=k}^n \binom{n}{i}\mathcal{G}_{a,b}(i,k;r)\prod_{j=0}^{n-i-1}(aj+(a+b)s).
\end{equation}
If $0 \leq k \leq n-m$, then
\begin{equation}\label{gt3e2}
\binom{k+m}{k}\mathcal{G}_{a,b}(n,k+m;r+s)=\sum_{i=k}^{n-m}\binom{n}{i}\mathcal{G}_{a,b}(i,k;r)\mathcal{G}_{a,b}(n-i,m;s).
\end{equation}
\end{theorem}

The $a=0,b=1$ case of the following identity is a refinement of the $r$-Bell number relation \cite[Theorem 2]{Me}.

\begin{theorem}\label{gt4}
If $n,m,k \geq 0$, then
\begin{equation}\label{gt4e1}
\mathcal{G}_{a,b}(n+m,k;r)=\sum_{i=0}^n\sum_{j=0}^m \binom{n}{i}\mathcal{G}_{a,b}(m,j;r)\mathcal{G}_{a,b}(i,k-j;0)\prod_{\ell=0}^{n-i-1}(a\ell+a(m+r)+b(j+r)).
\end{equation}
\end{theorem}
\begin{proof}
Given $\lambda \in \mathcal{L}_r(n+m,k)$, consider the number, $n-i$, of elements in $I=[m+r+1,n+m+r]$ that lie in a block containing an element of $[m+r]$ and the number, $j+r$, of blocks occupied by the members of $[m+r]$.  There are then $\mathcal{G}_{a,b}(m,j;r)$ possibilities regarding placement of the members of $[m+r]$. Once these positions have been determined, there are $\binom{n}{i}\prod_{\ell=0}^{n-i-1}(a\ell+a(m+r)+b(j+r))$ ways in which to choose and arrange the aforementioned elements of $I$.  Finally, the remaining elements of $I$ can be arranged in $\mathcal{G}_{a,b}(i,k-j;0)$ ways as none of them can belong to distinguished blocks.  Summing over all possible $i$ and $j$ gives \eqref{gt4e1}.
\end{proof}

If $n \geq 0$, then let $G_{a,b}(n;r)=\sum_{k=0}^n G_{a,b}(n,k;r)$.  Note that $G_{a,b}(n;0)$ reduces to the Bell number \cite[A000110]{Sl} when $a=0,b=1$ and to the sequence \cite[A000262]{Sl} when $a=b=1$.  Summing \eqref{gt3e1} and \eqref{gt4e1} over $k$ gives, respectively, the formulas
\begin{equation}\label{Bne1}
\mathcal{G}_{a,b}(n;r+s)=\sum_{i=0}^n \binom{n}{i}\mathcal{G}_{a,b}(i;r)\prod_{j=0}^{n-i-1}(aj+(a+b)s)
\end{equation}
and
\begin{equation}\label{Bne2}
\mathcal{G}_{a,b}(n+m;r)=\sum_{i=0}^n\sum_{j=0}^m \binom{n}{i}\mathcal{G}_{a,b}(m,j;r)\mathcal{G}_{a,b}(i;0)\prod_{\ell=0}^{n-i-1}(a\ell+a(m+r)+b(j+r)).
\end{equation}

Note that the $a=0,b=1,s=1$ case of \eqref{Bne1} occurs as \cite[Theorem 7.1]{Me0}, see also \cite[Theorem 1]{Mi}; moreover, the $a=0,b=1$ case of \eqref{Bne2} occurs as \cite[Theorem 2]{Me}.

We have the following additional recurrences satisfied by the $\mathcal{G}_{a,b}(n;r)$.

\begin{theorem}\label{gt5}
If $n \geq 0$, then
\begin{equation}\label{gt5e1}
\mathcal{G}_{a,b}(n;r)=\sum_{i=0}^n \binom{n}{i}\mathcal{G}_{a,b}(n-i;0)\prod_{j=0}^{i-1}(aj+(a+b)r)
\end{equation}
and
\begin{align}
\mathcal{G}_{a,b}(n+1;r)&=r\sum_{i=0}^n \binom{n}{i} \mathcal{G}_{a,b}(n-i;r-1)\prod_{j=0}^{i}(aj+a+b)\notag\\
&\quad +\sum_{i=0}^n \binom{n}{i}\mathcal{G}_{a,b}(n-i;r)\prod_{j=0}^{i-1}(aj+a+b).\label{gt5e2}
\end{align}
\end{theorem}
\begin{proof}
To show \eqref{gt5e1}, consider the number, $i$, of elements in $[r+1,r+n]$ that belong to distinguished blocks within $\mathcal{L}_r(n)=\cup_{k=0}^n \mathcal{L}_r(n,k)$.  Note that there are $\binom{n}{i}$ ways to select these elements and $\prod_{j=0}^{i-1}(aj+(a+b)r)$ ways in which to arrange them, once selected, within the distinguished blocks. (Note that the $j$-th smallest element chosen is the $j$-th to be arranged and thus contributes $a(j-1)+(a+b)r$ for $1 \leq j \leq i$.)  The remaining $n-i$ elements of $[r+1,r+n]$ may then be partitioned in $\mathcal{G}_{a,b}(n-i;0)$ ways.  Summing over all possible $i$ gives \eqref{gt5e1}.

To show \eqref{gt5e2}, we consider whether or not the element $n+r+1$ belongs to a distinguished block within a member of $\mathcal{L}_r(n+1)$.  If it does, then there are $r$ choices for the block, which we will denote by $B$.  If there are $i$ other elements of $[r+1,r+n+1]$ in $B$, then there are $\binom{n}{i}$ ways in which to select these elements and $\prod_{\ell=0}^i(a\ell+a+b)$ ways in which to arrange all $i+2$ elements within $B$.  The remaining $n-i$ elements of $[r+1,r+n+1]$ and the other $r-1$ elements of $[r]$ can then be arranged together according to any member of $\mathcal{L}_{r-1}(n-i)$.  Thus, the first sum on the right-hand side of \eqref{gt5e2} gives the weight of all members of $\mathcal{L}_r(n+1)$ in which $n+r+1$ belongs to a distinguished block.  By similar reasoning, the second sum gives the weight of all members of $\mathcal{L}_r(n+1)$ in which $n+r+1$ belongs to a non-distinguished block according to the number $i$ of other elements in this block.
\end{proof}

Taking $a=0,b=1$ in \eqref{gt5e1} gives
$$B_{n,r}=\sum_{i=0}^n r^i\binom{n}{i}B_{n-i}, \qquad n \geq 0,$$
which is equivalent to the $x=1$ case of \cite[Equation 4]{Me0}, where $B_{n,r}=\sum_{k=0}^n \Stirling{n}{k}_r$ denotes the $r$-Bell number and $B_n$ denotes the usual Bell number.

Taking $a=0,b=1$ in \eqref{gt5e2}, and applying \eqref{Bne1} when $s=1$ to both sums, gives
$$B_{n+1,r}=rB_{n,r}+B_{n,r+1}, \qquad n \geq 0,$$
which is \cite[Theorem 8.1]{Me0}.

\textbf{Remark:}  Adding a variable $x$ that marks the number of non-distinguished blocks within members of $\mathcal{L}_r(n)$, identity \eqref{gt5e2} can be generalized to
\begin{align*}
G_{a,b,x}(n+1;r)&=r\sum_{i=0}^n \binom{n}{i} \mathcal{G}_{a,b,x}(n-i;r-1)\prod_{j=0}^{i}(aj+a+b)\\
&\quad +x\sum_{i=0}^n \binom{n}{i}\mathcal{G}_{a,b,x}(n-i;r)\prod_{j=0}^{i-1}(aj+a+b),
\end{align*}
which reduces to \cite[Theorem 4.2]{Me0} when $a=0,b=1$.  The other identities above for $\mathcal{G}_{a,b}(n;r)$ can also be similarly generalized.

\section{Combinatorial proofs of $r$-Lah formulas}

In this section, we provide combinatorial proofs of the following relations involving the $r$-Lah numbers which were given in \cite[Theorem 3.11]{NR}.

\begin{theorem} Let $0 \leq k \leq n$ and $r,s \geq 0$.  Then
\begin{align*}
(i)& \quad\binom{n}{k}(2r-2s)^{\overline{n-k}}=\sum_{j=k}^n (-1)^{j-k}\Lah{n}{j}_r \Lah{j}{k}_s,\\
(ii)& \quad \stirling{n}{k}_{2r-s}=\sum_{j=k}^n (-1)^{j-k}\Lah{n}{j}_r \stirling{j}{k}_s, \quad \text{if} \quad 2r \geq s,\\
(iii)& \quad \Stirling{n}{k}_{2s-r}=\sum_{j=k}^n (-1)^{n-j} \Stirling{n}{j}_r \Lah{j}{k}_s, \quad \text{if} \quad 2s \geq r,\\
(iv)& \quad \Lah{n}{k}_{\frac{r+s}{2}}=\sum_{j=k}^n \stirling{n}{j}_r \Stirling{j}{k}_s, \quad \text{if r and s have the same parity}.
\end{align*}
\end{theorem}

\textbf{Proof of (i):}\\

First suppose $r \geq s$.  Consider the set of ordered pairs $(\alpha,\beta)$, where $\alpha \in \mathcal{L}_r(n,j)$ for some $k \leq j \leq n$ and $\beta$ is an arrangement of the $j+s$ blocks of $\alpha$ not containing the elements of $[s+1,r]$ according to some member of $\mathcal{L}_s(j,k)$. Note that within $\beta$, the blocks of $\alpha$ are ordered according to the size of the smallest elements.  Let $\mathcal{A}$ ($=\mathcal{A}_{n,k}$) denote the set of all such ordered pairs $(\alpha,\beta)$.  Define the sign of $(\alpha,\beta) \in \mathcal{A}$ by $(-1)^{j-k}$, where $j$ denotes the number of non-distinguished blocks of $\alpha$.  Then the right-hand side of (i) gives the sum of the signs of all members of $\mathcal{A}$.

Let $\mathcal{A}^*\subseteq \mathcal{A}$ comprise those pairs in which each block of $\beta$ contains only one block of $\alpha$, with this block being a singleton.  Then each member of $\mathcal{A}^*$ has positive sign and $|\mathcal{A}^*|=\binom{n}{k}(2r-2s)^{\overline{n-k}}$.  To show the latter statement, first note that the blocks of $\beta$ for each $(\alpha,\beta) \in \mathcal{A}^*$ contain $k$ elements of $[r+1,r+n]$, together with the members of $[s]$.  Thus, there are $\binom{n}{k}$ choices concerning the elements of $[r+1,r+n]$ to go in these blocks. The remaining $n-k$ elements of $[r+1,r+n]$ then belong to the blocks of $\alpha$ containing the members of $[s+1,r]$.  Note that these $n-k$ elements may be positioned in any one of $(2r-2s)^{\overline{n-k}}$ ways amongst these blocks, as there are $2r-2s+i-1$ ways to position the $i$-th smallest element for $1 \leq i \leq n-k$ (upon selecting the position first for the smallest element and then for the second smallest and so on). This implies the cardinality formula for $|\mathcal{A}^*|$ above.

We now define a sign-changing involution of $\mathcal{A}-\mathcal{A}^*$, which will complete the proof for the case $r \geq s$.  To do so, given $(\alpha,\beta) \in \mathcal{A}-\mathcal{A}^*$, suppose that the blocks of $\beta$ are arranged from left to right in increasing order according to the size of the smallest element of $[n+r]$ contained therein.  Identify the left-most block of $\beta$ containing at least two elements of $[n+r]$ altogether, which we will denote by $B$.  If the first block of $\alpha$ within $B$ is a singleton, whence $B$ contains at least two blocks of $\alpha$, then erase brackets and move the element contained therein to the initial position of the block that follows.  If the first block of $\alpha$ within $B$ contains at least two elements of $[n+r]$, then form a singleton block using the initial element which then becomes the first block within the sequence of blocks comprising $B$. One may verify that this mapping provides the desired involution, which completes the proof in the case when $r \geq s$.

If $r<s$, then we show equivalently
\begin{equation}\label{Lahe1}
\binom{n}{k}(2s-2r)^{\underline{n-k}}=\sum_{j=k}^n (-1)^{n-j}\Lah{n}{j}_r \Lah{j}{k}_s.
\end{equation}
In this case, we define ordered pairs $(\alpha,\beta) \in \mathcal{A}$, where $\alpha$ is as before and $\beta$ is an arrangement of all of the blocks of $\alpha$, together with $s-r$ singleton blocks $\{-1\}, \{-2\}, \ldots, \{-(s-r)\}$ (which we will refer to as \emph{special}), according to some member of $\mathcal{L}_s(j,k)$.  Note that the distinguished blocks of $\alpha$, together with the special blocks, are to be regarded as distinguished elements within $\beta$ (with similar terminology applied to the blocks of $\beta$).  Furthermore, let us refer to the blocks of $\beta$ containing $\{-i\}$ for some $i$ as \emph{special} and the other blocks of $\beta$ as \emph{non-special}.  Define the sign by $(-1)^{n-j}$, where $j$ is the number of non-distinguished blocks of $\alpha$.

Let $\mathcal{A}^*$ consist of all ordered pairs $(\alpha,\beta)$ in which all blocks of $\alpha$ are singletons and are distributed within $\beta$  such that the non-special blocks of $\beta$ contain only one block of $\alpha$, while the special blocks of $\beta$ have (i) at most one block of $\alpha$ to the right of the special singleton contained therein, and (ii) at most one to the left of it.  Then each member of $\mathcal{A}^*$ has positive sign and $|\mathcal{A}^*|=\binom{n}{k}(2s-2r)^{\underline{n-k}}$.  This follows from first observing that there are $\binom{n}{k}$ ways in which to choose and arrange the elements of $[r+1,r+n]$ that are not to be contained within the special blocks of $\beta$.  It is then seen that there are $(2s-2r)^{\underline{n-k}}$ ways in which to arrange the remaining elements of $[r+1,r+n]$ in the special blocks of $\beta$ according to the restrictions above.

To define the sign-changing involution of $\mathcal{A}-\mathcal{A}^*$, first apply the mapping used in the prior case if some non-special block of $\beta$ contains two or more elements of $[n+r]$ altogether.  Otherwise, identify the smallest $i \in [r-s]$, which we will denote by $i_0$, such that the block of $\beta$ containing $\{-i\}$ violates condition (i) or (ii).  Apply the involution used in the prior case to the blocks of $\alpha$ to the left of $\{-i_0\}$ within its block in $\beta$ if (i) is violated, and if not, then apply this mapping to the blocks of $\alpha$ occurring to the right of $\{-i_0\}$.  Combining the two mappings yields the desired involution of $\mathcal{A}-\mathcal{A}^*$ and completes the proof in the case when $r<s$. \qed \hfill\\

\textbf{Proof of (ii):}\\

In what follows, let $\mathcal{C}_r(n,k)$ denote the subset of $\mathcal{L}_r(n,k)$ enumerated by $\stirling{n}{k}_r$, i.e., those distributions in which the smallest element is first within each block.  First assume $r=s$.  In this case, let $\mathcal{B}$ ($=\mathcal{B}_{n,k}$) denote the set of ordered pairs $(\gamma,\delta)$ such that $\gamma \in \mathcal{L}_r(n,j)$ for some $k \leq j \leq n$ and $\delta$ is an arrangement of all the blocks of $\gamma$ according to some member of $\mathcal{C}_r(j,k)$.  Here, it is understood that the blocks of $\gamma$ are ordered according to the size of the smallest element, with the distinguished blocks of $\gamma$ considered distinguished as elements of $\delta$.  Furthermore, the first block of $\gamma$ within each cycle of $\delta$ is the smallest (i.e., it contains the smallest element of $[n+r]$ contained within all of the blocks in the cycle).  Define the sign of $(\gamma,\delta)$ as $(-1)^{j-k}$, where $j$ denotes the number of non-distinguished blocks of $\gamma$.  Then the right-hand side of (ii) when $r=s$ is the sum of the signs of all members of $\mathcal{B}$.

Let $\mathcal{B}^* \subseteq \mathcal{B}$ comprise those pairs $(\gamma, \delta) \in \mathcal{B}$ satisfying the conditions (i) within each block of $\gamma$, the first element is smallest, and (ii) no block of $\delta$ contains two or more blocks of $\gamma$.  Note that members of $\mathcal{B}^*$ must contain $k$ non-distinguished blocks for otherwise (ii) would be violated, whence each member of $\mathcal{B}^*$ has positive sign.  Furthermore, members of $\mathcal{B}^*$ are seen to be synonymous with members of $\mathcal{C}_r(n,k)$.

We define a sign-changing involution of $\mathcal{B}-\mathcal{B}^*$ as follows. Suppose that the cycles of $\delta$ within $(\gamma,\delta) \in \mathcal{B}-\mathcal{B}^*$ are arranged in increasing order according to the size of the smallest element of $[n+r]$ contained therein.   Let $B$ be the left-most cycle of $\delta$ that either contains at least two blocks of $\delta$ or contains a block of $\delta$ in which the first element fails to be the smallest.  Consider the first block $x$ within $B$.  If $x$ is of the form $\{a,\ldots,b,\ldots\}$, where $b$ is the smallest element of the block and $b \neq a$, then replace it within $B$ with the two blocks $\{b,\ldots\},\{a,\ldots\}$.  On the other hand, if the first element of $x$ is the smallest, whence $B$ contains at least two blocks, then write all elements of $x$ in order at the end of the second block of $B$.  This mapping provides the desired involution and establishes the result in the case $r=s$.

Now suppose $r<s\leq 2r$ and write $s=2r-\ell$ for some $0 \leq \ell \leq r-1$.  We modify the proof given in the prior case as follows.  Let $\mathcal{B}$ consist of all ordered pairs $(\gamma,\delta)$, where $\gamma \in \mathcal{L}_r(n,j)$ and $\delta$ is an arrangement of all the blocks of $\gamma$, together with the special singletons $\{-i\}$ for $i \in [r-\ell]$, arranged according to some member of $\mathcal{C}_s(j,k)$.  Blocks are ordered according to the size of the smallest elements contained therein and the cycles of $\delta$ are arranged as before.  The distinguished elements of $\delta$ are the distinguished blocks of $\gamma$, together with the special singletons.  Let $\mathcal{B}^*\subseteq \mathcal{B}$ consist of those pairs satisfying conditions (i) and (ii) above, where for (ii), we exclude from consideration cycles of $\delta$ containing $\{-i\}$ for some $i$.  Given $(\gamma,\delta) \in \mathcal{B}-\mathcal{B}^*$, apply the prior involution to the cycles of $\delta$ not containing the special singletons.

Let $\mathcal{B}'\subseteq \mathcal{B}^*$ consist of those $(\gamma,\delta)$ in which the numbers $\pm i$ for $i \in [r-\ell]$ all belong to singleton blocks of $\gamma$ each occupying its own cycle of $\delta$.  Note that $|\mathcal{B}'|=\stirling{n}{k}_{2r-s}$ since there are $\ell=2r-s$ distinguished cycles (i.e., those containing a block with an element of $[r-\ell+1,r]$
in it), with all cycles containing a single contents-ordered block whose first element is also the smallest.  To complete the proof in this case, we extend the involution to $\mathcal{B}^*-\mathcal{B}'$ as follows.  Let $i_0$ be the smallest $i \in [r-\ell]$ such that either (a) a cycle of $\delta$ containing $\{-i\}$ also has one or more elements of $[r+1,r+n]$ in it, or (b) a cycle of $\delta$ containing $\{-i\}$ has only that block in it, with $i$ not occurring as a singleton block of $\gamma$.

If (a) occurs and there are at least two elements of $[r+1,r+n]$ altogether in the cycle of $\delta$ containing $\{-i_0\}$, then apply the mapping used in the proof of the $r\geq s$ case of (i) above to the blocks of this cycle excluding $\{-i_0\}$.  Otherwise, if (a) occurs and there is only one element of $[r+1,r+n]$ in the cycle containing $\{-i_0\}$ or if (b) occurs, then replace one option with the other by either moving the element in the other block of the cycle containing $\{-i_0\}$ to the last position of the block containing $i_0$ within its cycle or vice-versa.   Combining the last two mappings provides the desired involution of $\mathcal{B}^*-\mathcal{B}'$ and completes the proof in the $r<s \leq 2r$ case.

Finally, if $r>s$, then consider ordered pairs $(\gamma,\delta)$, where $\delta$ is an arrangement in cycles of all the non-distinguished blocks of $\gamma$, together with those containing $i$ for some $i \in [s]$.  Apply the involution used in the $r=s$ case above, but this time excluding from consideration those blocks of $\gamma$ containing $i$ for some $i \in [s+1,r]$.  The set of survivors $(\gamma,\delta)$ of this involution then consists of all $(\gamma,\delta)$ where any block of $\gamma$ (other than one containing some $i \in [s+1,r]$) has its smallest element first, with the cycles of $\delta$ each containing one block of $\gamma$.  We add $r-s$ to all of the elements in $[r+1,r+n]$. Then to a block of $\gamma$ containing $i \in [s+1,r]$, we write $i+r-s$ at the front of it. For each $i$, split the block now containing $i+r-s$ and $i$ into two separate blocks starting with these elements.  Designate all blocks starting with $i \in [2r-s]$ as distinguished.  From this, it is seen that the set of survivors of the involution in this case are synonymous with members of $\mathcal{C}_{2r-s}(n,k)$, which completes the proof.  \qed \\

\textbf{Proof of (iii):}\\

One can give a similar proof to (ii) above.  We describe the main steps.  Let $\mathcal{S}_r(n,k)$ denote the subset of $\mathcal{L}_r(n,k)$ enumerated by $\Stirling{n}{k}_r$, i.e., those distributions in which the elements occur in increasing order within each block.  In the case $r=s$, consider the set $\mathcal{D}$ of ordered pairs $(\rho,\tau)$ such that $\rho \in \mathcal{S}_r(n,j)$ for some $j$ and $\tau$ is an arrangement of the blocks of $\rho$ according to some member of $\mathcal{L}_r(j,k)$. Define an involution of $\mathcal{D}$ by considering the first block of $\tau$ containing a non-singleton block of $\rho$ or in which the blocks of $\rho$ are not arranged in increasing order of smallest elements (possibly both).  Within this block of $\tau$, in a left-to-right scan of the blocks of $\rho$ contained therein, consider the first occurrence of either (i)  consecutive blocks of the form $B=\{x\}$, $C=\{y,\ldots\}$, where $x > \max(C)$, or (ii) $C=\{y,\ldots\}$, where $|C|\geq 2$ and the block directly preceding $C$ (if it exists) contains a single element $x$ that is strictly smaller than $\max(C)$.  We replace one option with the other by either moving the element in $B$ to the end of block $C$ in (i) or taking the last element of block $C$ as in (ii) and forming a singleton that directly precedes it.

If $r<s$, then add $s-r$ special singleton blocks to the arrangement $\tau$ to be regarded as distinguished.  Apply the involution used in the previous case to the blocks of $\tau$ not containing a special singleton.  To the blocks of $\tau$ containing a special singleton, we apply the involution separately to the sections to the left and to the right of it.  Given a survivor of this involution, we break the blocks of $\tau$ into two sections with the special singleton starting the second section and then add a distinguished element to the first section.  Note that this results in $r+2(s-r)=2s-r$ distinguished blocks in all.

If $s<r\leq2s$,  then consider ordered pairs $(\rho,\tau)$, where $\tau$ consists of contents-ordered blocks whose elements are the non-distinguished blocks of $\rho$, together with the first $s$ distinguished blocks of $\rho$.  Apply the involution used in the $r=s$ case, excluding from consideration those blocks of $\rho$ containing a member of $[s+1,r]$.  We extend this involution by considering the smallest $i \in [r-s]$, if it exists, such that there is at least one member of $[r+1,r+n]$ in either the block of $\tau$ containing $i$ or in the block of $\rho$ containing $r+1-i$ (possibly both).  Let $M$ denote the largest element of $[r+1,r+n]$ contained in either of these blocks.  If $M$ belongs to the block of $\tau$ containing $i$, necessarily as a singleton $\{M\}$, then erase the brackets enclosing $M$ and move it to the final position of the block of $\rho$ containing $r+1-i$, and vice-versa, if $M$ belongs to the block of $\rho$ containing $r+1-i$.  The set of survivors of this extended involution is seen to have cardinality $\Stirling{n}{k}_{2s-r}$.  \qed\\

\textbf{Proof of (iv):}\\

Suppose $r$ and $s$ have the same parity.  First assume $r \geq s$.  Let $\mathcal{E}$ denote the set of ordered pairs $(\alpha,\beta)$ such that $\alpha \in \mathcal{C}_r(n,j)$ for some $k \leq j \leq n$ and $\beta$ is an arrangement of the $j$ non-distinguished cycles of $\alpha$, together with $s$ special singleton cycles $(-1), (-2),\ldots, (-s)$ into $k+s$ blocks according to some member of $\mathcal{S}_s(j,k)$.  Then the right-hand side of (iv) gives the cardinality of $\mathcal{E}$.  To complete the proof in this case, we define a bijection between the sets $\mathcal{E}$ and $\mathcal{L}_{\frac{r+s}{2}}(n,k)$.

To do so, given $(\alpha,\beta) \in \mathcal{E}$, let $C_i$, $1 \leq i \leq r$, denote the cycles of $\alpha$ containing the members of $[r]$.  We assume that the smallest element is written first within a cycle.  If $1 \leq i \leq s$, then let $C_1^{(i)},C_2^{(i)},\ldots,C_{t_i}^{(i)}$ for some $t_i \geq 0$ denote the other cycles (if any) within the block of $\beta$ containing $(-i)$, arranged in decreasing order of smallest elements.  For each $i$, we remove the parentheses enclosing these cycles and concatenate the resulting words into one long word, which we write to the left of the elements in cycle $C_i$ in a single block.  (If there are no other cycles in the block of $\beta$ containing $(-i)$, then only the elements in cycle $C_i$ are written.)  This yields $s$ contents-ordered blocks containing the distinguished elements $1,\ldots,s$.

For $\frac{r+s}{2}<j \leq r$, consider the word $W_j$ obtained by reading the contents of cycle $C_j$ from left to right, excluding the initial letter $j$.  We then write the letters of $W_j$ in order, followed by the contents of cycle $C_{j-\frac{r-s}{2}}$, in a single block for each $j$.  This yields $\frac{r-s}{2}$ additional blocks containing the distinguished elements $s+1,\ldots,\frac{r+s}{2}$.  For the other $k$ blocks of $\beta$ (which contain cycles of $\alpha$ having only elements in $[r+1,r+n]$), express the permutation corresponding to the sequence of cycles contained therein as a word.  Putting these blocks together with the prior ones yields a Lah distribution containing all the elements of the set $[\frac{r+s}{2}]\cup[r+1,r+n]$ in which members of $[\frac{r+s}{2}]$ all belong to distinct blocks.  Subtracting $\frac{r-s}{2}$ from each letter  in $[r+1,r+n]$ yields a member of $\mathcal{L}_{\frac{r+s}{2}}(n,k)$, which we will denote by $f(\alpha,\beta)$.

To reverse the mapping $f$, given $L \in \mathcal{L}_{\frac{r+s}{2}}(n,k)$, we reconstruct its pre-image $(\alpha,\beta) \in \mathcal{E}$ as follows.  First observe that within the block of $L$ containing $i$ for some $i \in [s]$, a left-to-right minima (excepting $i$), taken together with the sequence of letters between it and the next minima, corresponds to a cycle belonging to the block of $\beta$ containing $(-i)$, with the letters to the right of and including $i$ forming the cycle $C_i$ of $\alpha$.  Within blocks of $L$ containing $i$ for $i \in [s+1,\frac{r+s}{2}]$, elements to the right of and including $i$ constitute cycle $C_i$ of $\alpha$, while those to the left of $i$ (if any) constitute the letters beyond the first of cycle $C_{i+\frac{r-s}{2}}$ in $\alpha$. Finally, writing the permutations in the undistinguished blocks of $L$ as cycles (and adding $\frac{r-s}{2}$ to each letter in $[\frac{r+s}{2}+1,\frac{r+s}{2}+n]$) yields the remaining cycles of $\alpha$ and blocks of $\beta$.

Now assume $r<s$ and let $\mathcal{E}$ consist of the ordered pairs $(\alpha,\beta)$ as before.  To define the mapping $f$ in this case, we proceed as follows.  For each $i \in [\frac{s-r}{2}+1,s]$, we delete the cycle $(-i)$ from its block within $\beta$ and then concatenate the contents of the remaining cycles, where cycles within a block are arranged in decreasing order of size of their first elements.  We then write the resulting word in a block followed by the contents of cycle $C_{i-\frac{s-r}{2}}$ if $\frac{s-r}{2}<i \leq \frac{s+r}{2}$, or followed by the contents of the cycles in the block containing the special cycle $(-(i-\frac{s+r}{2}))$ if $\frac{s+r}{2}<i \leq s$.  In the latter case, cycles within a block are arranged by decreasing order of first elements, except for the special cycle, which is first.

For each of the remaining blocks of $\beta$, we express the permutation corresponding to the sequence of cycles contained therein as a word. At this point, we have $k+\frac{s+r}{2}$ contents-ordered blocks of the set $S\cup [r+n]$ in which members of $S\cup[r]$ belong to distinct blocks, where $S=\{-1,-2,\ldots,-\frac{s-r}{2}\}$.  To each element of $[r+1,r+n]$, we add $\frac{s-r}{2}$, and to each element of $S$, we add $\frac{s+r}{2}+1$.  This results in a member of $\mathcal{L}_{\frac{r+s}{2}}(n,k)$ and the prior steps are seen to be reversible.  This completes the proof in the case $r<s$.  \qed\\

Modifying appropriately the proofs given above for (i) and (iv) in the case $r=s$ yields the following generalization in terms of $\mathcal{G}_{a,b}(n,k;r)$.

\begin{theorem}\label{genor}
If $0 \leq k \leq n$ and $r \geq 0$, then
\begin{equation}\label{genore1}
\delta_{n,k}=\sum_{j=k}^n (-1)^{j-k}\mathcal{G}_{a,b}(n,j;r)\mathcal{G}_{b,a}(j,k;r)
\end{equation}
and
\begin{equation}\label{genore2}
\mathcal{G}_{a,b}(n,k;r)=\sum_{j=k}^n \mathcal{G}_{a,t}(n,j;r)\mathcal{G}_{-t,b}(j,k;r).
\end{equation}
\end{theorem}

\begin{corollary}\label{genorc1}
Let $(a_n)_{n=0}^\infty$ and $(b_n)_{n=0}^\infty$ be sequences of complex numbers.  Then we have $b_n=\sum_{k=0}^n \mathcal{G}_{a,b}(n,k;r)a_k, ~ n \geq 0,$ if and only if $a_n=\sum_{k=0}^n(-1)^{n-k}\mathcal{G}_{b,a}(n,k;r)b_k, ~ n \geq 0.$
\end{corollary}


\end{document}